\documentclass[letterpaper, 10 pt, conference]{ieeeconf} 

\IEEEoverridecommandlockouts              

\let\labelindent\relax

\usepackage{fancyhdr}

\usepackage{setspace}
\usepackage{amsmath, amssymb, amsfonts}
\usepackage{graphicx}

\usepackage{verbatim}
\usepackage{url}

\usepackage{graphicx}
\usepackage{color}

\usepackage{amsmath}
\usepackage{multicol}

\usepackage[shortlabels]{enumitem}

\newcommand{\Int}{\operatorname{int}}

\newcommand{\real}{\operatorname{Re}}

\newcommand{\diag}{\operatorname{diag}}

\newcommand{\R}{\mathbb R}

\newcommand{\N}{\mathcal N}

\newcommand{\T}{\mathbb T}
\renewcommand{\S}{\mathbb S}
\renewcommand{\P}{\mathbb P}

\newcommand{\be}{\begin{equation}}
\newcommand{\ee}{\end{equation}}

\newcommand*\diff{\mathop{}\!\mathrm{d}}

\newcommand{\st}{\, | \,}

\newtheorem{Theorem}{Theorem}
\newtheorem{Fact}{Fact}

\newtheorem{Proposition}{Proposition}

\newtheorem{Definition}{Definition}
\newtheorem{Example}{Example}
\newtheorem{Problem}{Problem}
\newtheorem{Remark}{Remark}

 \IEEEoverridecommandlockouts

\begin{document}


\title{Is my system of ODEs $k$-cooperative?\thanks{The research  of MM is
		supported in part by research grants from the Israel Science Foundation and the US-Israel Binational Science Foundation.}}

\author{Eyal Weiss and Michael Margaliot\thanks{
		EW is with the Dept. of Computer Science, Bar-Ilan University, Ramat Gan, 5290002, Israel.
		MM (Corresponding Author) is with the School of Electrical Engineering,
		and the Sagol School of Neuroscience, 
		Tel-Aviv University, Tel-Aviv~69978, Israel.
		E-mail: \texttt{michaelm@eng.tau.ac.il}}}

\maketitle
\begin{center} 
								Version Date: March 20, 2020\\
\end{center}

\begin{abstract}
A   linear dynamical system is called positive if its flow maps the non-negative  orthant to itself. 
More precisely, it maps the set of vectors 
with zero sign variations to itself. 
A   linear dynamical system is
 called~$k$-positive if its flow maps the set of vectors with up to~$k-1$
sign variations 
 to itself.  

A nonlinear dynamical system is called~$k$-cooperative
if its variational system, which is a time-varying  linear  dynamical system, is~$k$-positive. 
These systems have special asymptotic properties. For example, it was recently shown that strong~$2$-cooperative systems satisfy
 a strong  Poincar\'{e}-Bendixson
property.

Positivity and~$k$-positivity
  are easy to verify in terms of   the sign-pattern of the 
	matrix in the dynamics. However, these sign conditions are not invariant under a
	coordinate transformation. A natural 
	question is   to determine if a given~$n$-dimensional
	system is $k$-positive up to a coordinate transformation.
	We study this problem for two special kinds of transformations: permutations and scaling by a  signature matrix. For any~$n\geq 4$ and~$k\in\{2,\dots, n-2\}$, we provide 
	a graph-theoretical necessary and sufficient condition 
	for~$k$-positivity up to such
	coordinate transformations. 

We describe an application of our results to a specific class of Lotka-Volterra systems. 
\end{abstract}


\section{Introduction}\label{sec:intro}
The state-variables in many mathematical models
represent quantities that can never attain negative values.
In Markov chains~\cite{modeling_master_eqn}
 the state-variables represent probabilities, 
 in compartmental systems~\cite{sandberg78} and biological occupancy models~\cite{rfm_chap,RFMRD} 
 they represent the density  in each compartment, and so on.

The continuous-time
linear system
\be\label{eq:ctls}
\dot x=Ax
\ee
 is called \emph{positive} if its flow maps the non-negative orthant to itself~\cite{farina2000}. In other words, 
for any non-negative  initial state   
all the state-variables remain non-negative for all time~$t\geq 0$. Since the difference between trajectories is also a trajectory, this can be stated as the following \emph{partial ordering property}: if~$a,b$ are two initial conditions with~$a\leq b$ (i.e.,~$a_i\leq b_i$ for all~$i$)
then~$x(t,a)\leq x(t,b)$ for all~$t\geq 0$. 
It is well-known that~\eqref{eq:ctls} is positive if and only if~(iff)~$A$ is Metzler, that is,~$a_{ij}\geq 0$ for any~$i\not =j$.

 Positive systems have important and special properties~\cite{posi-tutorial}. For example, a positive LTI that is  asymptotically stable always admits a \emph{diagonal} quadratic Lyapunov function~\cite{RANTZER201572}. Diagonal stability implies that certain associated nonlinear systems have a well-ordered behavior~\cite{diagsta_book,cyclic_pass}.

The linear system~\eqref{eq:ctls}
 is called~\emph{$k$-positive} if its flow maps 
the set of vectors with up to~$k-1$ sign variations to itself. Thus, $1$-positive systems are positive systems.
Ref.~\cite{eyal_k_pos}
derived  simple necessary and sufficient conditions
guaranteeing $k$-positivity for any~$k$. These turn out to be sign-pattern conditions on~$A$. 
For example,~\eqref{eq:ctls}
is~$2$-positive iff~$A$ has the following sign pattern:
	  \begin{align*}
		a_{1n}  & \leq 0, \;  a_{n1}   \leq 0,\\ 
	  a_{ij}&\geq 0 \text{ for all } i,j \text{ with } |i-j|=1,\\
		  a_{ij} &=  0 \text{ for all } i,j \text{ with }1<|i-j|<n-1.
\end{align*}
Note that such an~$A$ is not necessarily Metzler, as~$a_{1n},a_{n1}$ can be negative. 

Positivity has far reaching applications also to nonlinear systems.
Consider the  nonlinear system
\be\label{eq:nlon}
\dot x=f(x),
\ee
and assume that its solutions evolve
 on a convex state-space~$\Omega\subset\R^n$.
Let~$J(x):=\frac{\partial}{\partial x}f(x)$ denote the Jacobian of the vector field.
The 
nonlinear system 
is called \emph{cooperative} if its solutions
satisfy  the partial ordering property 
described above.
Recall (see e.g.~\cite{sontag_cotraction_tutorial})
 that if~$a,b$ are two initial conditions, and~$z(t):=x(t,a)-x(t,b)$ then
\be\label{eq:var}
\dot z(t)=M(t)  z(t) .
\ee
where
\be\label{eq:mtmat}
M(t):=\int_0^1 J(rx(t,a)+(1-r)x(t,b)) \diff r .
\ee
Eq.~\eqref{eq:var} describes how the variation between two solutions evolves in time. 
Thus,~\eqref{eq:nlon} is cooperative iff the LTV system~\eqref{eq:var} is positive for all~$a,b\in\Omega$ and all~$t\geq 0$.

The quasi-convergence theorem of Hirsch states 
that almost every  bounded solution of a strongly 
cooperative system converges to the set of equilibria~\cite{hlsmith}. 
Cooperative systems have found numerous applications is 
neuroscience, systems biology, and chemistry. Indeed, in these fields it is often the case that the effect of one state-variable on another is either excitatory or inhibitory, making it relatively easy to verify that the Jacobian~$J(x)$ is Metzler for all~$x$. Angeli and 
Sontag~\cite{mcs_angeli_2003}
introduced an important generalization of 
cooperative systems (and more generally monotone systems) to control systems, and derived a small gain theorem for the interconnection of such systems.

The nonlinear system~\eqref{eq:nlon} is called~\emph{$k$-cooperative}
if the variational system~\eqref{eq:var} is~$k$-positive for all~$a,b \in \Omega$ and all~$t\geq 0$. 
For~$k=1$ this reduces to a standard cooperative system.
It was recently shown that strongly~$2$-cooperative systems
satisfy a strong Poincar\'{e}-Bendixson
property: if~$x(t,a)$ is a bounded solution and its omega limit set~$\omega(a)$ does not include an equilibrium then~$\omega(a)$ is a periodic solution~\cite{eyal_k_pos}.
These results are closely related to the seminal work of  
Mallet-Paret and   Smith on monotone cyclic feedback systems~\cite{poin_cyclic}.

The conditions for~$k$-positivity
and~$k$-cooperativity 
 are not invariant under 
 coordinate transformations. 
For example, if~$\dot x=Ax$ and~$y:=Tx $, with~$T$ a nonsingular matrix,  then
$\dot y =TAT^{-1} y$, and~$TAT^{-1}$ may be Metzler
 even if~$A$ is not Metzler. For an LTI there is a well-known spectral condition guaranteeing that the flow maps a proper cone to itself~\cite{berman87}, but this condition does nor carry over naturally to the
 system~\eqref{eq:var}. Furthermore, the set of vectors with up to~$k-1$ sign variations is not a proper cone. 
Nevertheless, if~$J(x)$ satisfies some \emph{sign pattern condition}
for all~$x\in\Omega$ then clearly this carries over to~$M(t)$ in~\eqref{eq:mtmat}.

 This raises the following question. 
\begin{Problem}\label{prob:ma}
Given the 
dynamical system~\eqref{eq:nlon}   and a set of nonsingular
matrices~$\T \subseteq \R^{n\times n}$, is there a matrix~$T\in\T$ and an integer~$k$ 
such that~$TJ( x)T^{-1}$ satisfies the sign condition for~$k$-positivity for all~$x\in \Omega$? 
\end{Problem}

For the case of positive systems this question is well-known. The motivation for addressing it is stated for example in~\cite{PhysRevE.54.3135}:
``The essential point is
the following: if a system can be shown to possess a partial
ordering in some coordinate system, then all essential dynamical
properties of partially ordered systems (to be described)
will hold, regardless of reference frame''.

Recall that a matrix~$S\in \R^{n\times n}$ is called a \emph{signature matrix} if~$S$ is diagonal, and its diagonal elements are plus or minus~$1$.
Smith~\cite{diag_sca_smith}
 solved  Problem~\ref{prob:ma}
 for the particular case of
  cooperativity (i.e.~$1$-coopertivity) and 
transformations by a  signature matrix.

Here we address Problem~\ref{prob:ma} for two  types of transformations: permutations and scaling by a signature matrix.  For any~$k\in\{2,\dots,n-2\}$ 
we give a necessary and sufficient graph-theoretic condition  for $k$-cooperativity up to such transformations.
  
The notion of~$k$-positive systems has also been extended to discrete-time systems~\cite{rami_osci,rola_spect,rolda_dt}, but here we consider only
 the continuous-time case.

The remainder of this note is organized as follows. 
The next section reviews   known definitions and results that will be used later on. Section~\ref{sec:main} presents  the main results,
and Section~\ref{sec:lts} 
describes an application to nonlinear Lotka-Volterra systems.  

\section{Preliminaries}

We first briefly review   $k$-positive systems. For two integers~$i\leq j$, we use~$[i,j]$
to denote the set~$\{i,i+1,\dots,j\}$.  
 The non-negative orthant in~$\R^n$ is
$\R^n_+:=\{ x\in \R^n \st x_i \geq 0,\; i\in[1,n]\}$.
For a matrix~$A\in\R^{n\times m}$,~$A'$ denotes the transpose of~$A$.
\subsection{$k$-positive systems}
For a vector~$x\in\R^n\setminus\{0\}$, let~$s^-(x)$ denote the number of sign variations in~$x$ after deleting all
the zero entries, with~$s^-(0)$ defined as zero. For example, for~$n=6$ and
\be\label{eq:yexa}
y=\begin{bmatrix} 1 &0 &-2.5 & 0 &0 &3 \end{bmatrix}',
\ee
 $s^-(y)$
is the number of sign variations in the vector~$\begin{bmatrix} 1 &  -2.5 & 3 \end{bmatrix}'$, so $s^-(y)=2$.
Let~$s^+(x)$ denote the maximal possible number of sign variations in~$x$
after replacing every zero entry by either minus or plus one. For example, for~$y$ in \eqref{eq:yexa}
$s^+(y)$ is the number of sign variations in the vector 
$\begin{bmatrix} 1 &1 &-2.5 & 1 &-1 &3 \end{bmatrix}'$,
 so $s^+(y)=4$. Clearly,
\[
0\leq s^-(x)\leq s^+(x)\leq n-1 \text{ for all } x\in\R^n.
\]

For any~$k\in[1,n-1]$, let
\begin{align*}
P^k_- &:=\{ x\in\R^n  \st  s^-(x)\leq k-1 \},\\
P^k_+ &:=\{ x\in\R^n  \st  s^+(x)\leq k-1 \}.
\end{align*}
For example,~$P^1_- =\R^n_+ \cup (-\R^n_+)$ 
and~$P^1_+=\Int(P^1_-)$. More generally, it is not difficult to show that~$P^k_-$ is closed, and that~$P^k_+=\Int(P^k_-)$ for all~$k$~\cite{rolda_dt}.

Fix a time interval~$-\infty\leq t_0<t_1\leq\infty$.
Consider the linear time-varying (LTV) system:
\be\label{eq:ltv}
\dot x(t)=A(t)x(t), \quad x(t_0)=x_0,
\ee
where~$A(\cdot):(t_0,t_1)\to \R^{n\times n}$ is
a    locally  (essentially)  
bounded measurable matrix function of~$t$.
Pick~$k\in[1, n-1]$. The LTV~\eqref{eq:ltv}   is called~\emph{$k$-positive} it its flow maps~$P^k_-$ to itself, and \emph{strongly~$k$-positive} if its flow maps~$P^k_-\setminus\{0\}$ to~$P^k_+$.

We say that~$A(t)$ \emph{admits a sign pattern} if there exists a symbolic matrix~$\bar A\in\{*,+, -, 0\}^{n\times n}$ such that:
if~$\bar a_{ij}=0$ then~$a_{ij}(t)=0$ for almost all~$t \in (t_0,t_1)$, and
if~$\bar a_{ij}=+$ [$\bar a_{ij}=-$] then~$a_{ij}(t) \geq 0$ [$a_{ij}(t) \leq 0$] for almost all~$t \in (t_0,t_1)$.
Note that if~$\bar a_{ij}=*$ then there is no constraint  on~$a_{ij}(t)$, i.e.~$*$ denotes a ``don't care''.

\begin{Definition}\label{def:Wkset}~\cite{eyal_k_pos}
	Pick~$n\geq 4 $ and~$k\in[2,n-2]$. Let~$\bar A_k^n\in\{*,-, 0,+ \}^{n\times n}$ 
	denote the sign matrix with:
	\begin{enumerate}[(I)]
		\item $\bar a_{ii}=*$ for all~$i$; 
		\item if $k$ is even [odd] then~$\bar a_{1n},   \bar a_{n1} =-$ [$\bar a_{1n},   \bar a_{n1} =+$];
			\item $\bar a_{ij} = +$ for all~$i,j$ with~$|i-j|=1$;
		\item $\bar a_{ij} =  0$ for all~$i,j$ with~$1<|i-j|<n-1$.
	\end{enumerate}
\end{Definition}
  For example,
$ \bar A^4_2=
\begin{bmatrix}
*& +& 0& -\\
+ &* &+ & 0 \\
0& +&*& +\\
-& 0&  +&*
\end{bmatrix}.
$
Note that for~$k$ odd~$\bar A^n_k$ is in particular Metzler, but for~$k$ even~$\bar A^n_k$ is not necessarily Metzler. 
Also,~$\bar A^n_2=\bar A^n_4=\dots=\bar A^n_{2k}$
for all~$k$ such that~$2k\leq n-2$, and~$\bar A^n_3=\bar A^n_5=\dots=\bar A^n_{2k+1}$
for all~$k$ such that~$2k+1\leq n-2$.  
Note also that~$  A (t) $ satisfies the sign pattern of both~$\bar A^n_k$ and~$A^n_{k+1}$
 iff~$  A(t)$ is tridiagonal with non-negative entries on the super- and sub-diagonals for almost all~$t$.
An LTV system that satisfies  such a sign pattern is called a \emph{totally positive differential system}~\cite{schwarz1970}.
 Nonlinear systems   whose Jacobians satisfy such a sign pattern have been analyzed by Smillie~\cite{smillie}, Smith~\cite{periodic_tridi_smith}, and others (see the tutorial paper~\cite{margaliot2019revisiting} for more details).

\begin{Theorem}\label{thm:kiff}~\cite{eyal_k_pos}
Pick~$n\geq 4$ and~$k\in[2,n-2]$.
The system~\eqref{eq:ltv} is~$k$-positive on~$(t_0,t_1)$
iff~$A(t)$ admits the sign structure~$\bar A^n_k$ for almost all~$(t_0,t_1)$. 
\end{Theorem}

This implies 
that we can classify~$k$-positivity, with~$k\in[2,n-2]$
to just two cases: odd-positivity i.e. the flow of~\eqref{eq:ltv}
maps~$P^{k}$ to itself for all odd~$k\in[2,n-2]$,
and even-positivity i.e. the flow of~\eqref{eq:ltv}
maps~$P^{k}$ to itself for all even~$k\in[2,n-2]$.
  By definition,
	$
	P^1_-\subset P^2_- \subset \dots \subset P^{n-1},
	$
so this induces a Morse decomposition of the flow~\cite{GEDEON20091013}.

The next example demonstrates that the criteria
for~$k$-positivity described in Thm.~\ref{thm:kiff} are
 not invariant to coordinate transformations. 
\begin{Example}\label{ex:perm}
	Consider the symbolic matrix
	\be\label{eq:exabara}
	\bar A  = \begin{bmatrix}     
	*&0&0&+ \\
	+&*&+&0 \\
	0&  0&*&-\\
	+&0&-&*															
	\end{bmatrix}. 
	\ee
	Since~$\bar A \notin M^4_2$, 
	the corresponding system is in general 
	not even-positive. 
	However, for  the permutation matrix
	$
	P  := \begin{bmatrix}     
	0&0&1&0 \\
	0&1&0&0 \\
	1&0&0&0\\
	0&0&0&1																
	\end{bmatrix}, 
	$
	we have
	\[
	\bar B  :=  P\bar AP' = \begin{bmatrix}     
	*& 0&0&- \\
  + &*&+& 0 \\
	0&0&*&+\\
	- & 0& + &*														
	\end{bmatrix}, 
	\]
	so~$\bar B \in  M^4_2$. Thus, the system obtained by defining~$y(t):=Px(t)$ is even-positive.
\end{Example}

The effect of   certain coordinate   transformations
on the sign pattern of a matrix can be analyzed using a graph-theoretic approach. 

\subsection{Influence graphs}
We associate with a sign matrix~$\bar A\in \{*,0,+, -\}^{n\times n}$
a  signed and directed 
graph called the influence graph~$G=(V,E)$. The vertices are
$V=\{x_1,\dots,x_n\}$ (i.e., every vertex corresponds to a
 state-variable in the system~$\dot x=\bar A x$).
There is a directed edge from vertex~$x_j$ to~$x_i$ 
if~$ \bar a_{ij} \neq 0$ and $i \neq j$, and 
the sign of this edge is~$\bar a_{ij}$.
Note that there are no self-loops and that there 
it at most one edge from~$x_j$ to~$x_i$.
There is a one-to-one correspondence between the influence graph of a system and its sign pattern, 
except for the entries on the diagonal of~$\bar A$ (that are irrelevant for our considerations).
Thus, with a slight abuse of notation
  we say that~$\bar A$ satisfies some graph-theoretical property iff its
	influence graph satisfies  this property.

The set of in-neighbors [out-neighbors] of a vertex~$x_i$ is denoted by~$\N_{in}(x_i)$ [$\N_{out}(x_i)$],
and the set of neighbors of~$x_i$ is~$\N(x_i) := \N_{in}(x_i) \bigcup \N_{out}(x_i)$.
The \emph{in-degree} [\emph{out-degree}] 
of~$x_i$ is~$|\N_{in}(x_i)|$ [$|\N_{out}(x_i)|$], and the \emph{degree} is~$|\N(x_i)|$. 

For example, Fig.~\ref{fig:grapha}
depicts the influence graphs associated with the symbolic matrix~$\bar A$ and~$\bar B$ in Example~\ref{ex:perm}. 

From here on we always assume that
  the   influence graph is connected. Otherwise, we
can simply treat each connected component separately.

\section{Main Results}\label{sec:main}

Consider a symbolic matrix~$\bar A  $ that is either odd-positive,   even-positive, or both. 
Then the definition of the matrices~$M^n_k$ implies that
  its associated influence graph satisfies the following   properties:
\begin{enumerate}[(a)]
	\item \emph{Degree constraint:} 
	$|\N(x_i)| \leq 2$ for any vertex~$x_i$ in the graph. \label{item:two_neig}
	\item \emph{Sign-symmetric influence:}
	for any two vertices~$x_i$, $x_j$ in the graph,
	if the edge from
 $x_i$ to $x_j$ is~$+$ [$-$] then
there cannot be a~$-$ [$+$] edge  
 from~$x_j$ to~$x_i$. \label{item:SSI} 
\end{enumerate}
Furthermore,
	if $\bar A$ is   odd-positive then
\begin{enumerate}[(c)]
		\item 	\label{item:copt}
	  all edges are~$+$, 
		\end{enumerate}
	whereas if $\bar A$ is  not  odd-positive  and is even-positive 
	then \begin{enumerate}[(d)]
	\item  \label{item:dopt} there exist~$r\in[1,2]$ edges that are~$-$ and all other edges are~$+$.  
\end{enumerate}

\subsection{$k$-positivity up to permutations}\label{ssec:permute}
Our first two main results provide a necessary and sufficient  graph-theoretic condition for~$k$-positivity up to a permutation. Let~$\P\subset\R^{n\times n}$
 denote the set of  permutation
matrices.

\begin{Proposition}\label{prop:struct_k_pos_odd}
Let~$\bar A  $
be a symbolic matrix. 
The following two conditions are equivalent. 
\begin{enumerate} 
\item\label{item:1ftr}
There exist an odd~$k\in[2,n-2]$ and~$P\in \P$ such that~$P \bar A P'$ is~$k$-positive. 
\item \label{item:2frt}
 $\bar A$
satisfies  properties~\ref{item:two_neig},
\ref{item:SSI}, and~\ref{item:copt}.
\end{enumerate} 
\end{Proposition}

\begin{Proposition}\label{prop:struct_k_pos_even}
Let~$\bar A  $
be a symbolic matrix.
The following two conditions are equivalent. 
\begin{enumerate} [(I)]
\item\label{item:evenkk}
There exist an even~$k\in[2,n-2]$ and~$P\in \P$ such that~$P \bar A P'$ is~$k$-positive,
but there does not  exist ~$P\in \P$ such that~$P \bar A P'$ is
odd-positive. 
\item \label{item:eveniop}
 $\bar A$
satisfies  properties~\ref{item:two_neig},
\ref{item:SSI}, and~\ref{item:dopt}.
\end{enumerate}
\end{Proposition}

\begin{Example}\label{exa:refer}
The influence graph associated with the symbolic matrix~$\bar A$ in~\eqref{eq:exabara} 
  satisfies    properties~\ref{item:two_neig},
\ref{item:SSI}, and~\ref{item:dopt},
so by
Prop.~\ref{prop:struct_k_pos_even}
there exists~$P\in \P$ such that~$P \bar A P'$ is~$k$-positive. Furthermore, there does not exist
a~$P\in \P$ such that~$P \bar A P'$ is odd-positive.
\end{Example}

\begin{figure}[t]
	\centering
\begingroup%
\makeatletter%
\providecommand\color[2][]{%
	\errmessage{(Inkscape) Color is used for the text in Inkscape, but the package 'color.sty' is not loaded}%
	\renewcommand\color[2][]{}%
}%
\providecommand\transparent[1]{%
	\errmessage{(Inkscape) Transparency is used (non-zero) for the text in Inkscape, but the package 'transparent.sty' is not loaded}%
	\renewcommand\transparent[1]{}%
}%
\providecommand\rotatebox[2]{#2}%
\newcommand*\fsize{\dimexpr\f@size pt\relax}%
\newcommand*\lineheight[1]{\fontsize{\fsize}{#1\fsize}\selectfont}%
\ifx\svgwidth\undefined%
\setlength{\unitlength}{133.50949265bp}%
\ifx\svgscale\undefined%
\relax%
\else%
\setlength{\unitlength}{\unitlength * \real{\svgscale}}%
\fi%
\else%
\setlength{\unitlength}{\svgwidth}%
\fi%
\global\let\svgwidth\undefined%
\global\let\svgscale\undefined%
\makeatother%
\begin{picture}(1,0.42693019)%
\lineheight{1}%
\setlength\tabcolsep{0pt}%
\put(0,0){\includegraphics[width=\unitlength]{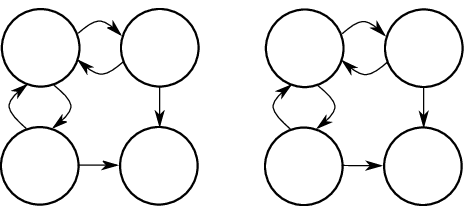}}%
\put(-0.39099016,0.72359639){\color[rgb]{0,0,0}\makebox(0,0)[lt]{\begin{minipage}{0.36112997\unitlength}\raggedright \end{minipage}}}%
\put(-1.13732551,2.15206618){\color[rgb]{0,0,0}\makebox(0,0)[lt]{\begin{minipage}{0.96301343\unitlength}\raggedright \end{minipage}}}%
\put(-0.19838745,1.87921241){\color[rgb]{0,0,0}\makebox(0,0)[lt]{\begin{minipage}{0.04815066\unitlength}\raggedright \end{minipage}}}%
\put(0.45967158,-0.32769318){\color[rgb]{0,0,0}\makebox(0,0)[lt]{\begin{minipage}{1.39636947\unitlength}\raggedright \end{minipage}}}%
\put(0.17859213,0.72259324){\color[rgb]{0,0,0}\makebox(0,0)[lt]{\begin{minipage}{0.36112997\unitlength}\raggedright \end{minipage}}}%
\put(-0.56774322,2.15106305){\color[rgb]{0,0,0}\makebox(0,0)[lt]{\begin{minipage}{0.96301344\unitlength}\raggedright \end{minipage}}}%
\put(0.37119485,1.87820927){\color[rgb]{0,0,0}\makebox(0,0)[lt]{\begin{minipage}{0.04815066\unitlength}\raggedright \end{minipage}}}%
\put(1.02925388,-0.32869635){\color[rgb]{0,0,0}\makebox(0,0)[lt]{\begin{minipage}{1.39636948\unitlength}\raggedright \end{minipage}}}%
\put(0.86172118,0.06841315){\color[rgb]{0,0,0}\makebox(0,0)[lt]{\lineheight{1.25}\smash{\begin{tabular}[t]{l}$x_2$\end{tabular}}}}%
\put(0.59931276,0.06833075){\color[rgb]{0,0,0}\makebox(0,0)[lt]{\lineheight{1.25}\smash{\begin{tabular}[t]{l}$x_1$\end{tabular}}}}%
\put(0.60072983,0.32794425){\color[rgb]{0,0,0}\makebox(0,0)[lt]{\lineheight{1.25}\smash{\begin{tabular}[t]{l}$x_4$\end{tabular}}}}%
\put(0.86306879,0.32652551){\color[rgb]{0,0,0}\makebox(0,0)[lt]{\lineheight{1.25}\smash{\begin{tabular}[t]{l}$x_3$\end{tabular}}}}%
\put(0.52941751,0.19627719){\color[rgb]{0,0,0}\makebox(0,0)[lt]{\lineheight{1.25}\smash{\begin{tabular}[t]{l}$-$\end{tabular}}}}%
\put(0.75307755,0.33626658){\color[rgb]{0,0,0}\makebox(0,0)[lt]{\lineheight{1.25}\smash{\begin{tabular}[t]{l}$+$\end{tabular}}}}%
\put(0.75494627,0.23529372){\color[rgb]{0,0,0}\makebox(0,0)[lt]{\lineheight{1.25}\smash{\begin{tabular}[t]{l}$+$\end{tabular}}}}%
\put(0.646138,0.19555901){\color[rgb]{0,0,0}\makebox(0,0)[lt]{\lineheight{1.25}\smash{\begin{tabular}[t]{l}$-$\end{tabular}}}}%
\put(0.74169487,0.10557795){\color[rgb]{0,0,0}\makebox(0,0)[lt]{\lineheight{1.25}\smash{\begin{tabular}[t]{l}$+$\end{tabular}}}}%
\put(0.83974355,0.2042255){\color[rgb]{0,0,0}\makebox(0,0)[lt]{\lineheight{1.25}\smash{\begin{tabular}[t]{l}$+$\end{tabular}}}}%
\put(0.292139,0.06941648){\color[rgb]{0,0,0}\makebox(0,0)[lt]{\lineheight{1.24999988}\smash{\begin{tabular}[t]{l}$x_2$\end{tabular}}}}%
\put(0.02973052,0.0693339){\color[rgb]{0,0,0}\makebox(0,0)[lt]{\lineheight{1.25}\smash{\begin{tabular}[t]{l}$x_1$\end{tabular}}}}%
\put(0.03114755,0.32894737){\color[rgb]{0,0,0}\makebox(0,0)[lt]{\lineheight{1.25}\smash{\begin{tabular}[t]{l}$x_4$\end{tabular}}}}%
\put(0.29348641,0.32752855){\color[rgb]{0,0,0}\makebox(0,0)[lt]{\lineheight{1.25}\smash{\begin{tabular}[t]{l}$x_3$\end{tabular}}}}%
\put(-0.04798931,0.19828351){\color[rgb]{0,0,0}\makebox(0,0)[lt]{\lineheight{1.25}\smash{\begin{tabular}[t]{l}$+$\end{tabular}}}}%
\put(0.19071853,0.33426028){\color[rgb]{0,0,0}\makebox(0,0)[lt]{\lineheight{1.25}\smash{\begin{tabular}[t]{l}$-$\end{tabular}}}}%
\put(0.1915839,0.24532514){\color[rgb]{0,0,0}\makebox(0,0)[lt]{\lineheight{1.25}\smash{\begin{tabular}[t]{l}$-$\end{tabular}}}}%
\put(0.0677281,0.19957161){\color[rgb]{0,0,0}\makebox(0,0)[lt]{\lineheight{1.25}\smash{\begin{tabular}[t]{l}$+$\end{tabular}}}}%
\put(0.17311554,0.10557785){\color[rgb]{0,0,0}\makebox(0,0)[lt]{\lineheight{1.25}\smash{\begin{tabular}[t]{l}$+$\end{tabular}}}}%
\put(0.26735308,0.20322243){\color[rgb]{0,0,0}\makebox(0,0)[lt]{\lineheight{1.25}\smash{\begin{tabular}[t]{l}$+$\end{tabular}}}}%
\end{picture}%
\endgroup%
	\caption{Influence graphs corresponding to~$\bar A$ (left) and~$\bar B$ (right) in Example~\ref{ex:perm}.}
	\label{fig:grapha}
\end{figure}
 
 {\sl Proof of Prop.~\ref{prop:struct_k_pos_odd}.}
If~\ref{item:1ftr}) holds then
the influence  graph associated  with~$P \bar A P'$
satisfies properties~\ref{item:two_neig}, \ref{item:SSI}, and~\ref{item:copt}. 
Since a permutation amounts to relabeling  the vertices, it is easy to see that this implies that the influence  graph associated  with~$  \bar A  $ also
satisfies properties~\ref{item:two_neig}, \ref{item:SSI}, 
and~\ref{item:copt}. We conclude that~\ref{item:1ftr})
implies~\ref{item:2frt}).

To prove the converse implication, assume that~\ref{item:2frt}) holds.
If~$\bar A$   is odd-positive
then~\ref{item:1ftr}) holds for~$P=I$. 
Thus, we may assume that~$\bar A$
  is not odd-positive.
	We  consider several cases.

\noindent {\sl{Case 1.}} 
Suppose 
 that every vertex has exactly two neighbors.
Pick an 
 arbitrary vertex, 
and associate with it a   state-variable~$y_1$.
Pick one of its neighbors and associate with it the
state-variable~$y_2$.
Now~$y_2$ has a  unique  neighbor that is not~$y_1$ and we denote it by~$y_3$.
 We proceed in this way choosing at each step  the only neighbor of the current vertex that has not  been chosen yet.
Since the influence graph is connected, this procedure ends after all vertices have been  indexed in increasing order.
The resulting influence graph is odd-positive. 
The~$y_i$s are clearly a permutation of the~$x_i$s 
i.e.,~$y=Px$, for some~$P\in\P$. 
 
\noindent  {\sl{Case 2.}}
 Suppose  that every vertex has exactly two neighbors, 
except for two vertices, say~$x_i$ and~$x_j$, that have a single  neighbor.
 Let~$y_1:=x_i$ and~$y_n:=x_j$. 
Since the graph is connected, 
 these vertices cannot be neighbors. 
Denote the single neighbor of~$y_1$ by~$y_2$.  
Denote the single neighbor of~$y_2$ that is not~$y_1$
by~$y_3$. 
 We proceed in this way choosing at each step  the only neighbor of the current vertex that has not  been chosen yet.
Since the influence graph is connected and all the vertices
except for~$y_1$ and~$y_n$ have two neighbors, 
this procedure ends after all vertices have been  indexed in increasing order.
The resulting influence graph is odd-positive (and also even-positive). 

Since we assume   that~\ref{item:2frt}) holds and the graph is connected, the two cases above are    the only 
possible cases, and in each such case we showed that~\ref{item:1ftr}) holds.
Thus,~\ref{item:2frt}) implies~\ref{item:1ftr}).~\hfill{$\QED$}

{\sl Proof of Prop.~\ref{prop:struct_k_pos_even}}.
Suppose that~\ref{item:evenkk} holds. Then there 
exists~$P\in\P$ such that~$P \bar A P'$ is even-positive, but not odd-positive. This implies that~$P \bar A P'$
satisfies properties~\ref{item:two_neig}, \ref{item:SSI}, and~\ref{item:dopt}. Indeed, there must be at least one~$-$ edge in~$P \bar A P'$, as otherwise it is 
also odd-positive.
Since a permutation amounts to relabeling  the 
vertices, $  \bar A  $ also
satisfies properties~\ref{item:two_neig}, \ref{item:SSI}, 
and~\ref{item:dopt}. We conclude that~\ref{item:evenkk} 
implies~\ref{item:eveniop}.

To prove the converse implication, assume that~\ref{item:eveniop} holds.
If~$\bar A$   is even-positive and not odd-positive 
then~\ref{item:evenkk} 
 holds for~$P=I$. 
Thus, we may assume this is not so.
	We  consider several cases.

\noindent {\sl{Case 1.}}
 Suppose that  all the edges are~$+$
except for two edges that are~$-$, and that 
  every vertex has exactly two neighbors.
	The sign-symmetric influence property 
implies that the~$-$ edges
 are incident to two vertices, say,~$x_i$ and~$ x_j$.
Let~$y_1:=x_i$, and~$y_n:=x_j$.
Choose the only neighbor of~$y_1$, that is not~$y_n$,
and associate with it the state-variable~$y_2$.
The rest of the construction is similar to that described in the proof of Prop.~\ref{prop:struct_k_pos_odd},
with the exception that the resulting graph is even-positive and not odd-positive.

\noindent {\sl{Case 2.}} 
Suppose that  all the edges are~$+$,
 except for a single  edge that is~$-$ 
and that every vertex has exactly two neighbors.
	 The sign-symmetric influence property 
implies that 
there exist  vertices~$x_i$ and~$x_j$ such that the
 edge from~$x_i$ to~$x_j$ is~$-$ and there is no edge from~$x_j$ to~$x_i$. The argument for this case is very similar
 to the one in the previous case.

\noindent {\sl{Case 3.}}
Suppose that all the edges are~$+$
 except for two edges that are~$-$
and that every vertex has exactly two neighbors except for two vertices, say,~$x_i$ and~$x_j$, that have a single neighbor.
Since the   graph is connected,~$x_i$ and~$x_j$ cannot be neighbors.
Thus, the negative edges are    incident to two 
vertices~$x_p, x_q$, that are different than~$x_i, x_j$.
Let~$y_1:=x_p$ and~$y_n:= x_q$. 
Next,    let~$y_2$ be the neighbour of~$y_1$ that is not~$y_n$.
We   iterate over the vertices labeling them as~$y_3$, $y_4$ and so on as  done in the previous cases.
This must end 
  at  
 either~$x_i$
or~$x_j$
as these   are the only vertices with
  a single neighbor.
   W.l.o.g. assume that this  is~$x_i$.
 Now, let~$y_{n-1}$ be the neighbour of~$y_n$ that is not~$y_1$. 
We   iterate over the remaining
 vertices,  labeling them as~$y_{n-2}$, $y_{n-3}$ and so on until we end up with~$x_j$, after which
all the state variables have been relabeled.
The resulting graph is even-positive, and not odd-positive.

\noindent {\sl{Case 4.}} Suppose that  all  the
edges are~$+$, except for one edge that is~$-$, 
and that every vertex has exactly two neighbors, except for two vertices that have a single neighbor.
The proof for this case is very similar 
 to the one in the previous case,
except  that now there is an edge from~$x_p$ to~$x_q$, but not from~$x_q$ to~$x_p$. 

The four cases above are  all the possible cases, and in each such case we showed that~\ref{item:evenkk} 
 holds.~\hfill{$\QED$}
   
Since the proofs of Props.~\ref{prop:struct_k_pos_odd}
and~\ref{prop:struct_k_pos_even} are constructive, they provide a way to actually 
determine a permutation that 
transforms~$\bar A$ into a~$k$-positive system. The next example demonstrates this. 
\begin{Example}
Consider again the matrix~$\bar A$
in  Example~\ref{ex:perm}. 
  This corresponds to Case~1 in the proof of Prop.~\ref{prop:struct_k_pos_even}.
Following the construction in this case yields the relabeling
\[
y_1=x_3,\; 
y_4=x_4, \;  y_2=x_2, \;  y_3=x_1,
\] 
or
\[ y_1=x_4, \; 
y_4=x_3, \;  y_2=x_2,\;   y_3=x_1.
\] 
The first of these is just the permutation used
 in Example~\ref{ex:perm} and applying this permutation
 indeed yields an even-positive system. It is easy to verify that this is true also for the permutation corresponding to the second relabeling. 
\end{Example}

The proofs of Props.~\ref{prop:struct_k_pos_odd}
and~\ref{prop:struct_k_pos_even} 
also provide  the  number of different 
permutations that yield a~$k$-positive system.
\begin{Fact}\label{fact:num_of_perm}
Suppose that~$\bar A$ is 
not~$k$-positive for any~$k\in[2,n-2]$.
If~$\bar A$  satisfies  properties~\ref{item:two_neig},
\ref{item:SSI}, and~\ref{item:dopt} then 
there exist~$2$ matrices in~$\P$ such that~$P\bar A P'$
is even-positive. 
If~$\bar A$  satisfies  properties~\ref{item:two_neig},
\ref{item:SSI}, and~\ref{item:copt} then 
there exist~$2n$ matrices in~$\P$ such that~$P\bar A P'$
is odd-positive. 
\end{Fact}
\begin{proof}
Suppose that~$\bar A$  satisfies  properties~\ref{item:two_neig},
\ref{item:SSI}, and~\ref{item:dopt}.
  After the relabeling
	the state-variables that correspond to vertices 
that are incident to negative edges must be the first and the last state-variables. 
Thus, there are two possible options
to index them. After this choice is made, the indexing
of all the other state-variables is predetermined as 
in the proof of Prop.~\ref{prop:struct_k_pos_even}.
Indeed,   
any other choice yields a symbolic matrix~$\bar B$
 that has a
non-zero entry~$\bar b_{ij}$
with~$1<|i-j|<n-1$. 
Hence there are~$2$ possible  permutations.

Now suppose that~$\bar A$  satisfies  properties~\ref{item:two_neig},
\ref{item:SSI}, and~\ref{item:copt}.
Due to symmetry, every state-variable can be chosen to be 
the first state-variable (i.e.~$y_1$) yielding~$n$ possibilities.
After this choice is made, the indexing of all the other state-variables
can be done in either increasing   or decreasing order. 
Thus, there are overall~$2n$ possible permutations.
\end{proof}

	Next we study Problem~\ref{prob:ma}
	when we allow also scaling by a signature matrix (also called sign transformations). Note that if~$S $
is a signature matrix then~$S^{-1}=S$.

\subsection{$k$-positivity up to permutations 
and sign-transformations}\label{ssec:sign_trans}
 Using   sign-transformations allows
to detect more~$k$-positive systems ``in disguise''. 
 We begin with a simple example demonstrating this. 

\begin{Example}\label{ex:sign}
	Consider the symbolic matrix
	\be\label{eq:matre}
	\bar A  = \begin{bmatrix}     
	* &0&0&+ \\
	- &*&- &0 \\
	0&0&*&-\\
	+ &0& - &*																
	\end{bmatrix}. 
	\ee
The associated influence graph has four~$-$ edges, 
so Props.~\ref{prop:struct_k_pos_odd} and~\ref{prop:struct_k_pos_even}
imply that for any~$k\in[2,n-2]$
there does not exist a~$P\in\P$
so that~$P\bar A P'$ is~$k$-positive. 
For   the signature matrix~$S=\diag \{1,-1,1,1\}$,
the matrix
\[
S\bar A S = \begin{bmatrix}     
	* &0&0&+ \\
	+ &*&+ &0 \\
	0&0&*&-\\
	+ &0&- &*																 			
	\end{bmatrix}, 
\]
  satisfies  properties~\ref{item:two_neig},
\ref{item:SSI}, and~\ref{item:dopt},
so Prop.~\ref{prop:struct_k_pos_even} implies that there exists a
permutation matrix~$P$ such that~$PS \bar  A SP'$
is even-positive.  
\end{Example}

Let~$\S\subset\R^{n\times n}$  denote the set of signature  
matrices.
We introduce the following definition. 
\begin{Definition}
We call the symbolic matrix~$\bar A$
\emph{structurally odd-even-positive} if there exist~$P_1\in\P$
and~$S_1\in\S$ such that~$P_1S_1\bar AS_1P_1'$
is odd-positive and   there exist~$P_2\in\P$ and~$S_2\in\S$ so 
that~$P_2S_2\bar AS_2P_2'$
is even-positive.
We call~$\bar A$
\emph{structurally even-positive}  [\emph{structurally odd-positive}]
if it is \emph{not} odd-even-positive, 
and there exist~$P\in\P,S\in\S$  so that~$PS\bar ASP'$
is even-positive [odd-positive].
\end{Definition}

If~$S\in\S$ then
 letting~$y:=Sx$
implies that either~$y_i =  x_i$ or~$y_i=-x_i$
for all~$i$.
 In the latter case, the effect on the 
influence graph is 
  flipping the signs of all  the edges 
incident to   vertex~$x_i$. This implies that
properties~\ref{item:two_neig} and~\ref{item:SSI} are 
invariant to  sign-transformations.
To study the effect of setting~$y_i=-x_i$
on properties~\ref{item:copt} and~\ref{item:dopt}, we introduce more  notation.
For a symbolic matrix~$\bar A$, 
let~$\zeta(\bar A)$ denote the number of pairs of neighbors~$(x_i, x_j)$ such that
there is an~$-$ edge from~$x_i$ to~$x_j$ and/or from~$x_j$ to~$x_i$.
For example, for the matrix~$\bar A$ in~\eqref{eq:matre},
we have that three such pairs:~$(x_1,x_2)$, $(x_2,x_3)$, and~$(x_3,x_4)$, so~$\zeta(\bar A)=3$.

\begin{Proposition}\label{prop:sign_similar_iff}
Let~$\bar A$ be a symbolic matrix.
The following two conditions are equivalent.
\begin{enumerate}
\item\label{item:imp11}
$\bar A$ is either structurally odd-even-positive or
 structurally odd-positive.
\item\label{item:imp22}
 $\bar A$
satisfies properties~\ref{item:two_neig} 
 and~\ref{item:SSI} 
 and~$\zeta(\bar A)$ is even.
	\end{enumerate}
\end{Proposition}

\begin{Proposition}\label{prop:even_similar_iff}
Let~$\bar A$ be a symbolic matrix.
The following two conditions are equivalent.
\begin{enumerate}[(I)]
\item\label{item:steven}
$\bar A$ is structurally   even-positive.
\item\label{item:zetaodd}
 $\bar A$
satisfies properties~\ref{item:two_neig} 
 and~\ref{item:SSI}  
 and~$\zeta(\bar A)$ is odd.
	\end{enumerate}
\end{Proposition}

\begin{Example}
	Consider the symbolic matrix~$\bar A$
	in Example~\ref{ex:sign}.
	It  
	satisfies properties~\ref{item:two_neig},~\ref{item:SSI}, 
	and~$\zeta(\bar A)=3$.
	Thus, Prop.~\ref{prop:even_similar_iff} implies
	that~$\bar A$ is structurally   even-positive.
\end{Example}

{\sl Proof of Prop.~\ref{prop:sign_similar_iff}}.
Suppose that~\ref{item:imp11}  holds, i.e.
  there exist~$P\in\P$  and~$S\in \S$ such that~$\bar B: = PS\bar ASP'$
is odd-positive. 
Then~$\bar B$ satisfies properties~\ref{item:two_neig}  and~\ref{item:SSI}, and thus so does~$\bar A$. The matrix~$\bar B$, and thus also the matrix~$P'\bar B P =  S \bar A S $,
    satisfies property~\ref{item:copt}.
		This   
		implies that~$\zeta(S \bar A S   )=0$.
Since negation of a state-variable~$x_i$
  flips the signs of all
the edges 
incident to the vertex~$x_i$,  $\zeta(\bar A)$ must be even.
Thus,~\ref{item:imp11}) implies~\ref{item:imp22}).

	To prove the converse implication, assume that~\ref{item:imp22}) holds. We consider two cases.
	
	\noindent \emph{Case 1.}
	If~$\zeta(\bar A)=0$ then all the edges are~$+$,
		so~$\bar A$ satisfies   
	properties~\ref{item:two_neig},~\ref{item:SSI}, and~\ref{item:dopt}, and 
	combining this with Prop.~\ref{prop:struct_k_pos_odd} implies that there exists~$P\in\P$
	such that~$P\bar A P'$ is odd-positive. 
	
		\noindent \emph{Case 2.}
	If~$\zeta(\bar A)>0$  
	then pick a vertex~$x_i$ that
	is incident to a~$-$ edge and apply a
negation on this vertex. If~$x_i$ was connected to 
all its neighbors (where the number of neighbors is either one or two) by~$-$ edges  then after the negation it
is connected to all of them by~$+$ edges. 
Thus,~$\zeta(S_i\bar AS_i)=\zeta(\bar A)-2$, where~$S_i\in\S$
corresponds to  the negation of~$x_i$. 
If~$\zeta(S_i\bar AS_i)=0$ 
then we conclude as in Case~1   
that there exists~$P\in\P$
	such that~$P S_i \bar A S_i P'$ is odd-positive. 
Else, we pick  another vertex that is incident to a negative edge and continue the process.

If~$x_i$ was connected to one neighbor~$x_j$ by an~$-$  edge and to another neighbor~$x_q$ by an~$+$ edge
then after the negation it is
 connected to~$x_q$ by an~$-$   edge.
We proceed by negating~$x_q$.
We continue  this procedure until~$\zeta$ decreases by~$2$.
This is bound to eventually take place, since every time we apply a negation neighbor after neighbor, 
we ``push'' the negative edges in the same ``direction'', until we finally apply a negation on a vertex 
that is connected to all its neighbors by negative edges (recall
that~$\zeta(\bar A)$ is positive and even).

Then  either there are no more negative edges in the graph,
or   we repeat the process again.
Since it is clear that we do not apply a negation on any state variable more than once, this process
is finite and terminates after decreasing~$\zeta$ to zero.
We conclude as in Case~1 that
that there exists~$P\in\P$
	such that~$PS\bar AS P'$ is odd-positive, where~$S\in\S$ represents all the negations used in the process.

	Thus, we showed that~\ref{item:imp22}) implies~\ref{item:imp11}).~\hfill{$\QED$}

{\sl Proof of Prop.~\ref{prop:even_similar_iff}.}
Suppose that~\ref{item:steven}  holds, i.e.
  there exist~$P\in\P$ and~$S\in \S$ such that~$\bar B: = PS\bar ASP'$ is even-positive, yet~$\bar A$
	is not structurally odd-even positive. 
Then~$\bar B$ satisfies properties~\ref{item:two_neig}  and~\ref{item:SSI}, and thus so does~$\bar A$. The matrix~$\bar B$, and thus also the matrix~$P'\bar B P =  S \bar A S $,
    satisfies property~\ref{item:dopt}. Combining this with the fact that~$\bar A$
	is not structurally odd-even positive implies 
	that~$\zeta(S \bar A S   )=1$.
Since negation of a state-variable~$x_i$
  flips the signs of all
the edges 
incident to the vertex~$x_i$,  $\zeta(\bar A)$ must be odd.
Thus,~\ref{item:steven}  implies~\ref{item:zetaodd}.

	To prove the converse implication, assume that~\ref{item:zetaodd}  holds.  
	Arguing as in the proof of Prop.~\ref{prop:sign_similar_iff}, 
	we can find an~$S\in\S$ such that~$\zeta(S \bar A S)=1$.
	This implies that~$r$, the number  of~$-$ edges,
	satisfies~$r\in\{1,2\}$. Thus,~$S\bar A S$ satisfies 
		properties~\ref{item:two_neig}, \ref{item:SSI}
		and~\ref{item:dopt}. Applying Prop.~\ref{prop:struct_k_pos_even} completes the proof of
	Prop.~\ref{prop:even_similar_iff}.~\hfill{$\QED$}

\begin{Remark}
	Since the proof of Props.~\ref{prop:sign_similar_iff}
	and~\ref{prop:even_similar_iff}
	are constructive, they provide 
	an algorithm for 
	finding the transformations that 
	  yield  a~$k$-positive system.	
\end{Remark}

\section{An application}\label{sec:lts}
Consider the nonlinear Lotka-Volterra system:
\be\label{eq:lts}
\dot x_i=x_i(r_i+\sum_{j=1}^n a_{ij}x_j),\quad i\in[1,n],
\ee
with~$r_i,a_{ij} \in\R$. Such systems play an important role in mathematical ecology~\cite{hofbauer_sigmund_1998}. 
The relevant  state-space in these 
applications is~$\R^n_+$.  
\begin{Proposition}\label{prop:LTS}
Consider~\eqref{eq:lts} with~$n\geq 4$ and such  
 that:
\begin{enumerate}
\item				for any $i \in [1,n]$ there are at most two values   $j\not =i$ such that $a_{ij}$ is non-zero; 
\item
$a_{ij} a_{ji}\geq 0$ for any~$i,j$.
\end{enumerate}
Then there exist~$P\in \P$, $S\in \S$ such that the system obtained by setting~$y(t):=P S x(t)$
 is~$k$-cooperative for some~$k\in[2,n-2]$.
\end{Proposition}

\begin{proof}
Let~$A:=\{a_{ij}\}_{i,j=1}^n$. For any~$i\not =j$, the~$(i,j)$ entry in the Jacobian of~\eqref{eq:lts} is~$J_{ij}(x)= a_{ij} x_i$. 
This implies that,  ignoring   diagonal entries 
(that are not relevant for our purposes), $J(x)$ has the same sign pattern as~$A$ for all~$x\in\R^n_+$.
The influence graph of~$A$ satisfies properties~\ref{item:two_neig} 
 and~\ref{item:SSI}. 
Combining this with
 Props.~\ref{prop:sign_similar_iff} and~\ref{prop:even_similar_iff} completes the proof. 
\end{proof}

Note that the fact that the nonlinear~$y$-system
is~$k$-cooperative has important implications. If~$k$ is odd then the system is~$1$-cooperative, i.e. cooperative. 
If~$k$ is even then the system is~$2$-cooperative and, under an additional irreducibility assumption,
it   satisfies the strong 
    Poincar\'{e}-Bendixson
property described in~\cite{eyal_k_pos}: any omega limit set of a bounded trajectory that does not include an equilibrium is a periodic orbit.

\section{Conclusion}\label{sec:conc}
 $k$-positive systems generalize the well-known  positive systems.
An important advantage of such systems is  the existence of simple sign-pattern conditions guaranteeing that a system is~$k$-positive. 
It was recently shown that     $k$-positive systems
provide a useful tool for analyzing the asymptotic behavior of 
 nonlinear systems. However, $k$-positivity is not invariant under coordinate transformations.

Here, we derived graph-theoretic 
necessary and sufficient conditions for 
a system to be~$k$-positive 
up to two types of transformations: permutations of the state-variables and scaling by a signature matrix.
We also   provided algorithms that explicitly find such transformations, when they exist.

It would be interesting to extend these results to  more general 
coordinate transformations.


\begin{thebibliography}{10}
\providecommand{\url}[1]{#1}
\csname url@samestyle\endcsname
\providecommand{\newblock}{\relax}
\providecommand{\bibinfo}[2]{#2}
\providecommand{\BIBentrySTDinterwordspacing}{\spaceskip=0pt\relax}
\providecommand{\BIBentryALTinterwordstretchfactor}{4}
\providecommand{\BIBentryALTinterwordspacing}{\spaceskip=\fontdimen2\font plus
\BIBentryALTinterwordstretchfactor\fontdimen3\font minus
  \fontdimen4\font\relax}
\providecommand{\BIBforeignlanguage}[2]{{%
\expandafter\ifx\csname l@#1\endcsname\relax
\typeout{** WARNING: IEEEtranS.bst: No hyphenation pattern has been}%
\typeout{** loaded for the language `#1'. Using the pattern for}%
\typeout{** the default language instead.}%
\else
\language=\csname l@#1\endcsname
\fi
#2}}
\providecommand{\BIBdecl}{\relax}
\BIBdecl

\bibitem{rolda_dt}
\BIBentryALTinterwordspacing
R.~Al-Seidi, M.~Margaliot, and J.~Garloff, ``Discrete-time $k$-positive linear
  systems,'' \emph{IEEE Trans.\ Automat.\ Control}, 2019, to appear. [Online].
  Available: \url{https://arxiv.org/abs/1910.08125}
\BIBentrySTDinterwordspacing

\bibitem{rola_spect}
R.~Alseidi, M.~Margaliot, and J.~Garloff, ``On the spectral properties of
  nonsingular matrices that are strictly sign-regular for some order with
  applications to totally positive discrete-time systems,'' \emph{J. Math.
  Anal. Appl.}, vol. 474, pp. 524--543, 2019.

\bibitem{sontag_cotraction_tutorial}
Z.~Aminzare and E.~D. Sontag, ``Contraction methods for nonlinear systems: A
  brief introduction and some open problems,'' in \emph{{Proc.\ 53rd IEEE Conf.
  on Decision and Control}}, Los Angeles, CA, 2014, pp. 3835--3847.

\bibitem{mcs_angeli_2003}
D.~Angeli and E.~D. Sontag, ``Monotone control systems,'' \emph{IEEE Trans.\
  Automat.\ Control}, vol.~48, pp. 1684--1698, 2003.

\bibitem{cyclic_pass}
M.~Arcak and E.~Sontag, ``Diagonal stability of a class of cyclic systems and
  its connection with the secant criterion,'' \emph{Automatica}, vol.~42, pp.
  1531--1537, 2006.

\bibitem{RFMRD}
E.~Bar-Shalom, A.~Ovseevich, and M.~Margaliot, ``Ribosome flow model with
  different site sizes,'' \emph{SIAM J. Applied Dynamical Systems}, vol.~19,
  no.~1, pp. 541--576, 2020.

\bibitem{berman87}
A.~Berman and R.~J. Plemmons, \emph{Nonnegative Matrices in the Mathematical
  Sciences}.\hskip 1em plus 0.5em minus 0.4em\relax SIAM, 1987.

\bibitem{farina2000}
L.~Farina and S.~Rinaldi, \emph{Positive Linear Systems: Theory and
  Applications}.\hskip 1em plus 0.5em minus 0.4em\relax John Wiley, 2000.

\bibitem{GEDEON20091013}
T.~Gedeon and G.~Hines, ``Multi-valued characteristics and {Morse}
  decompositions,'' \emph{J. Diff. Eqns.}, vol. 247, no.~4, pp. 1013--1042,
  2009.

\bibitem{modeling_master_eqn}
G.~Haag, \emph{Modelling with the Master Equation: Solution Methods and
  Applications in Social and Natural Sciences}.\hskip 1em plus 0.5em minus
  0.4em\relax Springer International Publishing, 2017.

\bibitem{hofbauer_sigmund_1998}
J.~Hofbauer and K.~Sigmund, \emph{The Theory of Evolution and Dynamical
  Systems: Mathematical Aspects of Selection}.\hskip 1em plus 0.5em minus
  0.4em\relax Cambridge University Press, 1988.

\bibitem{diagsta_book}
E.~Kaszkurewicz and A.~Bhaya, \emph{Matrix Diagonal Stability in Systems and
  Computation}.\hskip 1em plus 0.5em minus 0.4em\relax Boston, MA: Birkhauser,
  2000.

\bibitem{rami_osci}
\BIBentryALTinterwordspacing
R.~Katz, M.~Margaliot, and E.~Fridman, ``Entrainment to subharmonic solutions
  in oscillatory discrete-time systems,'' \emph{Automatica}, 2019, to appear.
  [Online]. Available: \url{https://arxiv.org/abs/1904.06547}
\BIBentrySTDinterwordspacing

\bibitem{PhysRevE.54.3135}
A.~S. Landsberg and E.~J. Friedman, ``Dynamical effects of partial orderings in
  physical systems,'' \emph{Phys. Rev. E}, vol.~54, pp. 3135--3141, 1996.

\bibitem{poin_cyclic}
J.~Mallet-Paret and H.~L. Smith, ``The {P}oincar{\'e}-{B}endixson theorem for
  monotone cyclic feedback systems,'' \emph{J. Dyn. Differ. Equ.}, vol.~2,
  no.~4, pp. 367--421, 1990.

\bibitem{margaliot2019revisiting}
M.~Margaliot and E.~D. Sontag, ``Revisiting totally positive differential
  systems: A tutorial and new results,'' \emph{Automatica}, vol. 101, pp.
  1--14, 2019.

\bibitem{posi-tutorial}
A.~Rantzer and M.~E. Valcher, ``A tutorial on positive systems and large scale
  control,'' in \emph{{Proc.\ 57th IEEE Conf. on Decision and Control}}, Miami
  Beach, FL, USA, 2018, pp. 3686--3697.

\bibitem{RANTZER201572}
A.~Rantzer, ``Scalable control of positive systems,'' \emph{European J.
  Control}, vol.~24, pp. 72--80, 2015.

\bibitem{sandberg78}
I.~W. Sandberg, ``On the mathematical foundations of compartmental analysis in
  biology, medicine, and ecology,'' \emph{IEEE Trans. Circuits and Systems},
  vol.~25, no.~5, pp. 273--279, 1978.

\bibitem{schwarz1970}
B.~Schwarz, ``Totally positive differential systems,'' \emph{Pacific J. Math.},
  vol.~32, no.~1, pp. 203--229, 1970.

\bibitem{smillie}
J.~Smillie, ``Competitive and cooperative tridiagonal systems of differential
  equations,'' \emph{SIAM J. Math. Anal.}, vol.~15, pp. 530--534, 1984.

\bibitem{periodic_tridi_smith}
H.~L. Smith, ``Periodic tridiagonal competitive and cooperative systems of
  differential equations,'' \emph{SIAM J. Math. Anal.}, vol.~22, no.~4, pp.
  1102--1109, 1991.

\bibitem{hlsmith}
H.~L. Smith, \emph{Monotone Dynamical Systems: An Introduction to the Theory of
  Competitive and Cooperative Systems}, ser. Mathematical Surveys and
  Monographs.\hskip 1em plus 0.5em minus 0.4em\relax Providence, RI: Amer.
  Math. Soc., 1995, vol.~41.

\bibitem{diag_sca_smith}
\BIBentryALTinterwordspacing
H.~L. Smith, ``Is my system of {ODEs} cooperative?'' 2012. [Online]. Available:
  \url{https://math.la.asu.edu/~halsmith/identifyMDS.pdf}
\BIBentrySTDinterwordspacing

\bibitem{eyal_k_pos}
\BIBentryALTinterwordspacing
E.~Weiss and M.~Margaliot, ``A generalization of linear positive systems with
  applications to nonlinear systems: Invariant sets and the
  {P}oincar\'{e}-{B}endixson property,'' 2019, submitted. [Online]. Available:
  \url{https://arxiv.org/abs/1902.01630}
\BIBentrySTDinterwordspacing

\bibitem{rfm_chap}
Y.~Zarai, M.~Margaliot, and T.~Tuller, ``Modeling and analyzing the flow of
  molecular machines in gene expression,'' in \emph{Systems Biology},
  N.~Rajewsky, S.~Jurga, and J.~Barciszewski, Eds.\hskip 1em plus 0.5em minus
  0.4em\relax Springer, 2018, pp. 275--300.

\end{thebibliography}

 \end{document}